\documentclass{amsart}
\usepackage{amsthm,amsmath,amsfonts,amssymb,amscd,mathrsfs,graphics}
\usepackage{latexsym}
\usepackage{graphicx}

\usepackage{epic}

\usepackage{enumerate}
\usepackage[...]{youngtab}

\usepackage{placeins}
\usepackage{tikz}

\newtheorem{thm}{Theorem}[section]

\newtheorem{lem}[thm]{Lemma}
\newtheorem{prop}[thm]{Proposition}

\makeatletter
\renewcommand{\@seccntformat}[1]{\S{\csname
the#1\endcsname}\hspace{0.5em}}
\makeatother

\begin{document}

\title{Strong Gelfand Pairs of SL($\bf{2,p}$)}

\author{Andrea Barton, Stephen P.  Humphries}
  \address{Department of Mathematics,  Brigham Young University, Provo, 
UT 84602, U.S.A.
E-mail: andreabarton24@gmail.com, steve@mathematics.byu.edu}
\date{}
\maketitle

\begin{abstract}  
A strong Gelfand pair $(G,H)$ is a group $G$ together with a subgroup $H$ such that every irreducible character of $H$ induces a multiplicity-free character of $G$.  We classify the strong Gelfand pairs of the special linear groups
$SL(2,p)$ where $p$ is a prime.

\medskip

\medskip
\end{abstract}

\theoremstyle{plain}

\theoremstyle{definition}
\newtheorem*{dfn}{Definition}
\newtheorem{exa}[thm]{Example}
\newtheorem{rem}[thm]{Remark}

\newcommand{\ds}{\displaystyle}
\newcommand{\bs}{\boldsymbol}
\newcommand{\mb}{\mathbb}
\newcommand{\mc}{\mathcal}
\newcommand{\mf}{\mathfrak}
\renewcommand{\mod}{\operatorname{mod}}
\newcommand{\mult}{\operatorname{Mult}}

\def \a{\alpha} \def \b{\beta} \def \d{\delta} \def \e{\varepsilon} \def \g{\gamma} \def \k{\kappa} \def \l{\lambda} \def \s{\sigma} \def \t{\theta} \def \z{\zeta}

\numberwithin{equation}{section}

\setlength{\leftmargini}{1.em} \setlength{\leftmarginii}{1.em}
\renewcommand{\labelenumi}{\setlength{\labelwidth}{\leftmargin}
   \addtolength{\labelwidth}{-\labelsep}
   \hbox to \labelwidth{\theenumi.\hfill}}

\maketitle

\section{Introduction}

Given a finite group $G$, a {\it multiplicity-free character of $G$} is a character $\chi$ of $G$ such that for every irreducible character $\psi$ of $G$, we have $\langle \chi, \psi \rangle \leq 1$.
 
A {\it Gelfand pair} $(G,H)$ is a group $G$ together with a subgroup $H$ such that the trivial character of $H$ induces a multiplicity-free character of
$G$. The importance of Gelfand pairs is indicated by the following equivalent conditions for finite groups (\cite{AHN}, \cite{Wik}):

\noindent (i) The algebra of $(H,H)$-double invariant functions on G with multiplication defined by convolution is commutative.

\noindent (ii) For any irreducible representation $\rho$ of $G$, the space of $H$-invariant vectors in $\rho$ is no more than 1-dimensional.

\noindent (iii) For any irreducible representation $\rho$ of $G$, the dimension of ${\rm Hom}_H(\rho, \mathbf 1)$ is less than or equal to 1, where $\mathbf 1$ denotes the trivial representation.

\noindent (iv) The permutation representation of $G$ on the cosets of $H$ is multiplicity-free.

\noindent (v) The centralizer algebra of the permutation representation is commutative.

\noindent (vi) The double cosets $HgH$ determine a commutative Schur ring over $G$.

A {\it strong Gelfand pair} $(G,H)$ is a finite group $G$ together with a subgroup $H$ such that every irreducible character of $H$ induces a multiplicity-free character of $G$. In other words, for any $\chi \in \hat{G}$, $\psi\in \hat{H}$, where $\hat{G}$ is the set of all irreducible characters of $G$ and $\hat{H}$ is the set of all irreducible characters of $H$, we have $\langle \chi, \psi \uparrow ^G \rangle \leq 1$. We note that a strong Gelfand pair is also referred to using the phrase “multiplicity one property” or “multiplicity one theorem” (see \cite{A,AA,B}).

Equivalently, $(G,H)$ is a strong Gelfand pair if and only if the Schur ring determined by the $H$-classes $g^H = \{ g^h : h \in H\}, g\in G$, is commutative \cite{Har}.

The strong Gelfand subgroups of the symmetric group were found in \cite{AHN}; those for wreath products of the form $F\wr S_n$, where $F$ is a finite abelian group, were found in \cite{Can}.

Let $G = {\rm SL}(2,p)$ and let $U = U(2,p) = \mathcal{C}_{p-1} \rtimes \mathcal{C}_p \leq G$ be the subgroup of upper triangular matrices. When $p$ is odd, we let $H_2 \leq U$ be the unique subgroup of $U$ of index $2$. The main result of this paper is to determine the strong Gelfand pairs of ${\rm SL}(2,p)$:

\begin{thm}\label{Thm1}
Let $p>11$ be prime.

\noindent (i) If $p\equiv 1 \mod 4$, then there is only one strong Gelfand pair  $(G,U)$.

\noindent (ii) If $p \equiv 3 \mod 4$, then there are two strong Gelfand pairs: $(G,U)$ and $(G,H_2)$.
\end{thm}

In Theorem \ref{Thm3} we classify the strong Gelfand pairs for primes $p \leq 11$.

\medskip
\noindent
{\bf Acknowledgment} All computations were made using Magma \cite{Mag}.

\section{Character Tables}
\FloatBarrier

In this section we give the character tables for $G$, $PSL(2,p)$, and $U$. This involves finding the character table for certain semidirect products. Let $p$ be a prime.

From Suzuki \cite[p. 393]{Suz}, the conjugacy classes of $U$ have the following representatives:

\noindent (i) $1 = \begin{pmatrix}
1 & 0 \\
0 & 1 \end{pmatrix}$;

\noindent (ii) $-1 = \begin{pmatrix}
-1 & 0 \\
0 & -1 \end{pmatrix}$;

\noindent (iii) Let $a = \begin{pmatrix}
\lambda & 0 \\
0 & 1/\lambda \end{pmatrix}$, where $\lambda$ is a generator of $\mathbb F_p^*$.
Then each $a^i$ represents a distinct class for $1\leq i < p-1$, $i\neq (p-1)/2$;

\noindent (iv) $b = \begin{pmatrix}
1 & 1 \\
0 & 1 \end{pmatrix} $;

\noindent (v) $d = \begin{pmatrix}
1 & t \\
0 & 1 \end{pmatrix} $, where $t$ is a non-square mod $p$;

\noindent (vi) $-b = \begin{pmatrix}
-1 & -1 \\
0 & -1 \end{pmatrix} $;

\noindent (vii) $-d = \begin{pmatrix}
-1 & -t \\
0 & -1 \end{pmatrix} $.

Thus there are $p+3$ classes.  The sizes of these classes are:

\noindent (i) 1;
\qquad (ii) 1;
\qquad (iii) $p$;
\qquad (iv) $(p-1)/2$;

\noindent (v) $(p-1)/2$;
\qquad (vi) $(p-1)/2$;
\qquad   (vii) $(p-1)/2$.

Observe that $U = \langle a, b \rangle$. We will use this to develop the character table.

Let $B = \langle b \rangle$, the subgroup of $U$ whose diagonal entries are 1. Then $B \triangleleft U$, so we lift the characters of $U / B \cong \langle a \rangle$ to get $p-1$ linear characters $1_U, \chi _{0,k}$ for $1\leq k \leq p-2$ which are trivial on $B$. Let $\zeta$ be a primitive $(p-1)^{\text{th}}$ root of unity. These linear characters are determined by their values on the generators $a$ and $b$:
$$
\chi_{0,k} (a) = \zeta^{k}, \quad
\chi_{0,k} (b) = 1.
$$

Because $U$ is a split metacyclic group, we can use a result of Munkholm \cite{Mun} to complete the character table as shown below.

Let $B' = \langle -1, b\rangle$. To find the nonlinear characters of $U$, we induce the linear characters of $B'$. These linear characters are given by
$$
T_{i, k} (b) = \eta ^i,
\quad
T_{i, k} (-1) = (-1)^k,
$$
where $\eta$ is a primitive $p^{\text{th}}$ root of unity, $1 \leq i < p$, and $k = 1, 2$. Recall that $\lambda$ is the generator of $\mathbb{F}_p^*$ used in the definition of $a$. The characters are induced as follows, from \cite[p. 458]{Mun}:

\begin{align*}
\chi_{i,k}((-1)^{\nu}b^\tau)
&= \sum_{\ell=0}^{(p-3)/2} T_{i,k}( a^{\ell} ((-1)^{ \nu }b^\tau)a^{- \ell}  ) 
= \sum_{\ell=0}^{(p-3)/2} T_{i,k}( (-1)^{ \nu }b^{\tau \lambda^{2 \ell}}  )
\\&= \sum_{\ell=0}^{(p-3)/2} (-1)^{\nu k} \eta^{i \tau \lambda^{2 \ell}}
= (-1)^{\nu k} \sum_{\ell=0}^{(p-3)/2} \eta^{i \tau \lambda^{2 \ell}}.
\end{align*}

Let
$$Z = \chi_{1,1}(b) =  \Sigma_{\ell = 0}^{(p-3)/2} \eta^{\lambda ^{2\ell}}$$
and
$$Z^{(t)} = \chi_{1,1}(b^t) = \Sigma_{\ell = 0}^{(p-3)/2} \eta^{t \lambda ^{2\ell}},$$
where $t$ is the non-square mod $p$ used in the definition of $d$. 

When $k$ is a square $\mod p$, each $T_{i,k}$ induces the character $\chi_{i,1}$. Additionally, when $k$ is a non-square $\mod p$, each $T_{i,k}$ induces the character $\chi_{i,2}$. By the Quadratic Gauss sum formula \cite[p. xxvii]{Atl}:

\begin{lem}\label{Lem1}
If $p \equiv 1 \mod 4$, then $Z=\frac 1 2 (-1+\sqrt{p}),$ and $Z^{(t)}=\frac 1 2 (-1-\sqrt{p})$.

\noindent If $p \equiv 3 \mod 4$, then $Z=\frac 1 2 (-1+i\sqrt{p})$ and  $Z^{(t)}=\frac 1 2 (-1-i\sqrt{p})$.
\end{lem}

Then we have
$$   \chi_{1,1}(- b)=-Z
  \quad  \text{and}\quad   \chi_{1,1}( - b^{t})=-{Z}^{(t)}.
$$  
One then sees that
$$
\chi_{1,2}(b) = Z, \quad
\chi_{1,2}(- b) = Z, \quad
\chi_{1,2}(b^{t}) = {Z}^{(t)}, \quad
\text{and} \quad
\chi_{1,2}( - b^{t})={Z}^{(t)}.
$$
This gives the following character table for $U$:

\begin{table}[ht]
\begin{center}\begin{tabular}{c|l|l|l|l|l|l|l}
& 1 & $-1$ & $a^j$ & $b$ & $d$ & $-b$ & $-d$ \\
\hline
$1$ & 1 & 1 & 1 & 1 & 1 & 1  & 1 \\[2mm]
$\chi_{0,k}$ & $1$ & $(-1)^k$ & $\zeta^{kj}$ & 1 & 1 & $(-1)^k$ & $(-1)^k$\\[2mm]
$\chi_{1,1}$ & $(p-1)/2$ & $-(p-1)/2$ & 0 & $Z$ & $Z^{(t)}$ & $-Z$ & $-Z^{(t)}$ \\[2mm]
$\chi_{1,2}$ & $(p-1)/2$ & $(p-1)/2$ & $0$ & $Z$ & $Z^{(t)}$ & $Z$ &$ Z^{(t)}$ \\[2mm]
$\chi_{2,1}$ & $(p-1)/2$ & $-(p-1)/2$ & 0 & $Z^{(t)}$ & $Z$ & $-Z^{(t)}$ & $-Z$ \\[2mm]
$\chi_{2,2}$ & $(p-1)/2$ & $(p-1)/2$ & $0$ & $Z^{(t)}$ & $Z$ & $Z^{(t)}$ & $Z$ 
\end{tabular}\end{center}
\medskip
\caption{Character Table of $U$}
\label{CTU}
\end{table} 

We continue using notation from above, and now let $\sigma$ be a primitive $(p+1)^{\text{th}}$ root of 1. Then a result of Frobenius \cite[p. 30]{Hum} tells us that the (ordinary) complex character table of $G$, with $p\equiv 1 \mod 4$, is as shown in Table \ref{CTG1} for $1\le i \le (p-3)/2$,\ $1 \le j \le (p-1)/2$,\ $1 \le l \le (p-3)/2$, and $1 \le m \le (p-1)/2$.
 
\begin{table}[ht]
\begin{center}\begin{tabular}{c|l|l|l|l|l|l|l|l}
& 1 & -1 & $b$ & $d$ & $-b$ & $-d$ & $a^l$ & $f^m$\\
\hline
$1$ & 1 & 1 & 1 & 1 & 1 & 1  & 1 & 1\\[2mm]
$\varphi$ & $p$ & $p$ & 0 & 0 & 0 & 0 & 1 & $-1$\\[2mm]
$\chi_i$ & $p+1$ & $(-1)^i(p+1)$ & 1 & 1 & $(-1)^i$ & $(-1)^i$ & $\zeta^{il} + \zeta^{-il}$ & 0\\[2mm]

$\theta_j$ & $p-1$ & $(-1)^j (p-1)$ & $-1$ & $-1$ & $(-1)^{j+1}$ & $(-1)^{j+1}$ & $0$ & $-(\sigma^{j m} + \sigma^{-j m})$\\[2mm]

$\xi_1$ & $(p + 1)/2$ & $(p+1)/2$ & $-Z^{(t)}$ & $-Z$ & $-Z^{(t)}$ & $-Z$ & $(-1)^l$ & 0\\[2mm]
$\xi_2$ & $(p + 1)/2$ & $(p+1)/2$ & $-Z$ & $-Z^{(t)}$ & $-Z$ & $-Z^{(t)}$ & $(-1)^l$ & 0\\[2mm]
$\eta_1$ & $(p - 1)/2$ & $-(p-1)/2$ & $Z$ & $Z^{(t)}$ & $-Z$ & $-Z^{(t)}$ & 0 & $(-1)^{m+1}$\\[2mm]
$\eta_2$ & $(p - 1)/2$ & $-(p-1)/2$ & $Z^{(t)}$ & $Z$ & $-Z^{(t)}$ & $-Z$ & 0 & $(-1)^{m+1}$
\end{tabular} \end{center}
\medskip
\caption{Character Table of $G$, with $p \equiv 1 \mod 4$}
\label{CTG1}
\end{table} 
 \medskip

In the case $p\equiv 3 \mod 4$, the final four characters are different, as shown in Table \ref{CTG3}, where $Z$ and $Z^{(t)}$ for $p\equiv 3 \mod 4$ are as given above.
\begin{table}
\begin{center}\begin{tabular}{c|l|l|l|l|l|l|l|l}
& 1 & -1 & $b$ & $d$ & $-b$ & $-d$ & $a^l$ & $f^m$\\
\hline
$\xi_1'$ & $(p + 1)/2$ & $-(p+1)/2$ & $-Z^{(t)}$ & $-Z$ & $Z^{(t)}$ & $Z$ & $(-1)^l$ & 0\\[2mm]
$\xi_2'$ & $(p + 1)/2$ & $-(p+1)/2$ & $-Z$ & $-Z^{(t)}$ & $Z$ & $Z^{(t)}$ & $(-1)^l$ & 0\\[2mm]
$\eta_1'$ & $(p - 1)/2$ & $(p-1)/2$ & $Z$ & $Z^{(t)}$ & $Z$ & $Z^{(t)}$ & 0 & $(-1)^{m+l}$\\[2mm]
$\eta_2'$ & $(p - 1)/2$ & $(p-1)/2$ & $Z^{(t)}$ & $Z$ & $Z^{(t)}$ & $Z$ & 0 & $(-1)^{m+l}$
 \end{tabular} \end{center}
 \medskip
\caption{Characters of $G$, with $p \equiv 3 \mod 4$}
\label{CTG3}
\end{table}

From the character table of ${\rm SL}(2,p)$ (Tables \ref{CTG1}, \ref{CTG3}) we deduce the following character table for ${\rm PSL}(2,p)$, which can also be found in \cite{Dor}; see Table \ref{CTP1}, where $1\le i \le (p-3)/2$,\ $1 \le j \le (p-1)/2$,\ $1 \le l \le (p-3)/2$, and $1 \le m \le (p-1)/2$. Here, we also require $i$ and $j$ to be even.

\begin{table}[]
\begin{center}\begin{tabular}{c|l|l|l|l|l}
& 1 & $b$ & $d$ & $a^l$ & $f^m$\\
\hline
$1$ & 1 & 1 & 1 & 1 & 1\\[2mm]
$\overline{\varphi}$ & $p$ & 0 & 0 & 1 & $-1$\\[2mm]
$\overline{\chi_i}$ & $p+1$ & 1 & 1 & $\zeta^{il} + \zeta^{-il}$ & 0\\[2mm]
$\overline{\theta_j}$ & $p-1$ & $-1$ & $-1$ & $0$ & $-(\sigma^{j m} + \sigma^{-j m})$\\[2mm]
$\overline{\xi_1}$ & $(p + 1)/2$ & $-Z^{(t)}$ & $-Z$ & $(-1)^l$ & 0\\[2mm]
$\overline{\xi_2}$ & $(p + 1)/2$ & $-Z$ & $-Z^{(t)}$ & $(-1)^l$ & 0\\[2mm]
\end{tabular} \end{center}
\medskip
\caption{Character table for ${\rm PSL}(2,p)$ with $p \equiv 1 \mod 4$}
\label{CTP1}
\end{table}

In the case $p\equiv 3 \mod 4$, the last two characters $\overline{\xi_1}, \overline{\xi_2}$ are different, as indicated in Table \ref{CTP3}.
\begin{table}
\begin{center}\begin{tabular}{c|l|l|l|l|l}
& 1 & $b$ & $d$ & $a^l$ & $f^m$\\
\hline
$\overline{\eta_1'}$ & $(p - 1)/2$ & $Z$ & $Z^{(t)}$ & 0 & $(-1)^{m+l}$\\[2mm]
$\overline{\eta_2'}$ & $(p - 1)/2$ & $Z^{(t)}$ & $Z$ & 0 & $(-1)^{m+l}$
 \end{tabular} \end{center}
 \medskip
\caption{Characters of ${\rm PSL}(2,p)$, with $p \equiv 3 \mod 4$}
\label{CTP3}
\end{table}

\FloatBarrier
\section{Group Theory Lemmas}

We collect here results that will be used in the proof of Theorem \ref{Thm1}.

\begin{lem}\label{Lem2}
Suppose $H\leq K\leq G$ are groups. If $(G,K)$ is not a strong Gelfand pair, then neither is $(G,H)$.
\end{lem}

\begin{proof}
Suppose $(G,K)$ is not a strong Gelfand pair.
So there are $\chi \in \hat G, \psi \in \hat K$ with $\langle \chi, \psi \uparrow ^G \rangle \geq 2$.
Let $\varphi$ be an irreducible character of H with $\langle \varphi \uparrow ^K, \psi \rangle > 0.$
Because $(\varphi \uparrow ^K)\uparrow ^G = \varphi \uparrow ^G$ and by Frobenius reciprocity, we have: 

\begin{align*}
\langle \varphi \uparrow ^G, \chi \rangle &= \langle (\varphi \uparrow ^K)\uparrow ^G, \chi \rangle = \langle \varphi \uparrow ^K, \chi |_K \rangle \\&\geq \langle \varphi \uparrow ^K , \psi \rangle \cdot \langle \psi, \chi |_K \rangle > 1.
\end{align*}
Thus, $(G,H)$ is not a strong Gelfand pair. 
\end{proof}

\begin{lem}\label{Lem3}
Suppose $N\leq H\leq G$ are groups with $N\triangleleft G$. Then $(G,H)$ is a strong Gelfand pair if and only if $(G/N, H/N)$ is a strong Gelfand pair.
\end{lem}

\begin{proof}
This result follows from the fact that the lift of an irreducible character is irreducible \cite[Theorem 17.3]{JaL}.
\end{proof}

For a finite group $G$, the {\it total character} $\tau$ of $G$ is the sum of all the irreducible characters of $G$ \cite{Tot2, Tot1, Tot3}.

\begin{lem}\label{Lem4}
Let $H\leq G$ and let $\tau$ be the total character of $H$. If there is some character $\chi$ of $G$ with $\deg (\tau) < \deg (\chi)$, then $(G,H)$ is not a strong Gelfand pair.
\end{lem}

\begin{proof}
Let $\varphi_1, \cdots, \varphi_r$ be the irreducible characters of $H$. Then for any character $\chi$ of $G$ we have
$$\chi|_H = \Sigma_{i=1}^r \epsilon _i \varphi_i,$$
where each $\epsilon_i\in \mathbb{N}$. Assume $(G,H)$ is a strong Gelfand pair. So, $\epsilon_i \leq 1$ for all $i$. Thus 
$$\deg(\chi) = \Sigma_{i=1}^r \epsilon_i \deg(\varphi_i) \leq \Sigma_{i=1}^r \deg(\varphi_i) = \deg(\tau).$$
We conclude that whenever $\deg(\tau) < \deg(\chi)$, it must be that $(G,H)$ is not a strong Gelfand pair.
\end{proof}

\begin{lem}\label{Lem5}
There is a one-to-one correspondence between maximal subgroups of ${\rm SL}(2,p)$ and maximal subgroups of ${\rm PSL}(2,p)$.
\end{lem}

\begin{proof}
By the lattice isomorphism theorem, this is equivalent to showing each maximal subgroup of $G$ contains $Z(G) = \langle -1 \rangle$. By way of contradiction, suppose there is a maximal subgroup $M$ of $G$ such that $-1 \notin M$. Then $a\notin M$, because $a^{(p-1)/2} = -1$. But then, $a\notin \langle -1, M \rangle$, so $M$ was not maximal.
\end{proof}

\section{Proof of Theorem \ref{Thm1}}

We first slightly restate Theorem \ref{Thm1}:

\begin{thm}\label{Thm2} Let $G = {\rm SL}(2,p)$ and let $U\le G$ denote the upper triangular subgroup. Let $H_2\le U$ denote the subgroup of elements whose diagonal entries are (non-zero) squares in $\mathbb F_p$.

For prime $p>11$, if $p\equiv 1 \mod 4$, then there is only one strong Gelfand pair  $(G,U)$, while if $p \equiv 3 \mod 4$, then there are two, $(G,U)$ and $(G,H_2)$.
\end{thm}

The cases for $p<13$ are given in Theorem \ref{Thm3}.

In this section, we now assume $p \geq 13$.

By Lemma \ref{Lem2}, we only need to consider maximal subgroups of ${\rm SL}(2,p)$. To do this, we use Lemma \ref{Lem5}. By Suzuki \cite[p. 417]{Suz}, up to conjugacy the maximal subgroups of ${\rm PSL}(2,p)$ are isomorphic to:

\noindent (i) a dihedral group of order $p+1$ or $p-1$;

\noindent (ii) the image of $U$ in ${\rm PSL}(2,p)$;

\noindent (iii) the alternating group $A_4$;

\noindent (iv) the symmetric group $\Sigma_4$; or

\noindent (v) the alternating group $A_5$.

From Table \ref{CTP1}, recall that ${\rm PSL}(2,p)$ has an irreducible character of degree $p+1$. We use Lemma \ref{Lem4} to show that the subgroups (conjugate to) (i), (iii), (iv), and (v) do not form strong Gelfand pairs with ${\rm PSL}(2,p)$. We will then show (ii) does form a strong Gelfand pair with ${\rm PSL}(2,p)$, and consider its maximal subgroups.
\medskip

\noindent {\bf Case (i):} Suppose $(p+1)/2$ is odd. The total character of the dihedral group of order $p+1$ has degree $2 + 2((p+1)/2 - 1)/2$ = $1 + (p+1)/2 < p+1$ for all primes. Thus by Lemma \ref{Lem4}, this dihedral group does not form a strong Gelfand pair with ${\rm PSL}(2,p)$.

Now suppose $(p+1)/2$ is even. The total character of the dihedral group of order $p+1$ has degree $4 + 2((p+1)/4 - 1) = 2 + (p+1)/2 < p+1$ for all $p > 5$. Thus, Lemma \ref{Lem4} again gives the result.

On the other hand, suppose $(p-1)/2$ is odd. The total character of the dihedral group of order $p-1$ has degree $2 + 2((p-1)/2 - 1)/2 = 1 + (p-1)/2 < p+1$ for all primes. Thus, Lemma \ref{Lem4} gives the result for this case.

Finally, suppose $(p-1)/2$ is even. The total character of the dihedral group of order $p-1$ has degree $4 + 2((p-1)/4 - 1) = 2 + (p-1)/2 < p+1$ for all primes. The result follows from Lemma \ref{Lem4}.
\medskip

\noindent {\bf Case (iii):} The total character of $A_4$ has degree $6$, which is less than $p+1$ for all $p > 5$. Thus, by Lemma \ref{Lem4}, this case does not produce a strong Gelfand pair.
\medskip 

\noindent {\bf Case (iv):} Similarly, the total character of $\Sigma_4$ has degree $10$, which is less than $p+1$ for all $p > 7$. It follows from Lemma \ref{Lem4} that $(\rm PSL(2,p), \Sigma_4)$ is not a strong Gelfand pair.
\medskip

\noindent {\bf Case (v):} Finally, the total character of $A_5$ has degree $16$ which is less than $p+1$ for all primes $p > 13$. For the case $p=13$, we see that $60$ does not divide the order of ${\rm SL}(2,13)$. So $A_5$ is not a subgroup of ${\rm SL}(2,13)$. Thus, Lemma \ref{Lem4} gives the result.\medskip

We have now shown that if a subgroup $H$ of ${\rm SL}(2,p)$ forms a strong Gelfand pair, $H$ must be a subgroup of $U$. To this end, we next show that $(G,U)$ is a strong Gelfand pair. This is given by the following result.

\begin{prop}\label{Prop1} We have the following decompositions of the characters of $G$ restricted to $U$:

\noindent (i) $(1_{G})|_{U} = 1_{U}$;

\noindent (ii) $\varphi|_{U} = 1_{U} + \chi_{1,2} + \chi_{2,2}$;

\noindent (iii) if $i$ is odd, $\chi_i|_{U} = \chi_{0,i} + \chi_{0,(p-1-i)} + \chi_{1,1} + \chi_{2,1}$;

\noindent (iv) if $i$ is even, $\chi_i|_{U} = \chi_{0,i} + \chi_{0,(p-1-i)} + \chi_{1,2} + \chi_{2,2}$;

\noindent (v) if $j$ is odd, $\theta_j|_{U} = \chi_{1,1} + \chi_{2,1}$;

\noindent (vi) if $j$ is even, $\theta_j|_{U} = \chi_{1,2} + \chi_{2,2}$;

\noindent (vii) $\xi_1|_{U} = \chi_{0,(p-1)/2} + \chi_{1,2}$;

\noindent (viii)  $\xi_2|_{U}= \chi_{0,(p-1)/2}+\chi_{2,2}$;

\noindent (ix) $\eta_1|_{U} = \chi_{1,1}$;

\noindent (x) $\eta_2|_{U} = \chi_{2,1}$;

\noindent (xi) $\xi_1'|_{U} = \chi_{0,(p-1)/2} + \chi_{1,1}$;

\noindent (xii) $\xi_2'|_{U} = \chi_{0,(p-1)/2} + \chi_{2,1}$;

\noindent (xiii) $\eta_1'|_{U} = \chi_{1,2}$;

\noindent (xiv) $\eta_2'|_{U}=  \chi_{2,2}$.

\end{prop}

\begin{proof}
We note that each side of any of the above equations is a character. Thus to show that the two sides are equal, we take the sum of the characters given on the right and show this equals the restriction given on the left.

We will frequently use the fact that $Z + Z^{(t)} = -1$ (see Lemma \ref{Lem1}). 

\noindent {\bf(i):}
Clearly the restriction of $1_{G}$ to $U$ gives $1_{U}$.

\noindent {\bf(ii):} This case follows from the calculation below:

\begin{center}\begin{tabular}{c|l|l|l|l|l|l|l}
& 1 & $-1$ & $a^j$ & $b$ & $d$ & $-b^{-1}$ & $-d^{-1}$ \\
\hline
$1_{U}$ & 1 & 1 & 1 & 1 & 1 & 1  & 1 \\[2mm]
$\chi_{1,2}$ & $(p-1)/2$ & $(p-1)/2$ & $0$ & $Z$ & $Z^{(t)}$ & $Z$ &$ Z^{(t)}$ \\[2mm]
$\chi_{2,2}$ & $(p-1)/2$ & $(p-1)/2$ & $0$ & $Z^{(t)}$ & $Z$ & $Z^{(t)}$ & $Z$ \\[2mm]
\hline
$\varphi |_U$ & $p$ & $p$ & 1 & 0 & 0 & 0 & 0
\end{tabular}\end{center}
   
\noindent{\bf(iii):} Here $i$ and $p$ are both odd. Thus, $(-1)^i = -1$ and $(-1)^{p-1-i} = -1$. Additionally, $\zeta^{(p-i-1)j} = \zeta^{(p-1)j}\zeta^{-ij} = \zeta^{-ij}$. This case follows from:

\begin{center}\begin{tabular}{c|l|l|l|l|l|l|l}
& 1 & $-1$ & $a^j$ & $b$ & $d$ & $-b^{-1}$ & $-d^{-1}$ \\
\hline
$\chi_{0,i}$ & $1$ & $-1$ & $\zeta^{ij}$ & 1 & 1 & $-1$ & $-1$\\[2mm]
$\chi_{0,(p-1-i)}$ & $1$ & $-1$ & $\zeta^{-ij}$ & 1 & 1 & $-1$ & $-1$\\[2mm]
$\chi_{1,1}$ & $(p-1)/2$ & $-(p-1)/2$ & 0 & $Z$ & $Z^{(t)}$ & $-Z$ & $-Z^{(t)}$ \\[2mm]
$\chi_{2,1}$ & $(p-1)/2$ & $-(p-1)/2$ & 0 & $Z^{(t)}$ & $Z$ & $-Z^{(t)}$ & $-Z$ \\[2mm]
\hline
$\chi _i |_U$ & $p+1$ & $-(p+1)$ & $\zeta^{ij} + \zeta^{-ij}$ & 1 & 1 & -1 & -1
\end{tabular}\end{center}

\noindent{\bf(iv):} Here $i$ is even and $p$ is odd. Thus, $(-1)^i = 1$ and $(-1)^{p-1-i} = 1$. In this case, we have:

\begin{center}\begin{tabular}{c|l|l|l|l|l|l|l}
& 1 & $-1$ & $a^j$ & $b$ & $d$ & $-b^{-1}$ & $-d^{-1}$ \\
\hline
$\chi_{0,i}$ & $1$ & $1$ & $\zeta^{ij}$ & 1 & 1 & $1$ & $1$\\[2mm]
$\chi_{0,(p-1-i)}$ & $1$ & $1$ & $\zeta^{-ij}$ & 1 & 1 & $1$ & $1$\\[2mm]
$\chi_{1,2}$ & $(p-1)/2$ & $(p-1)/2$ & $0$ & $Z$ & $Z^{(t)}$ & $Z$ &$ Z^{(t)}$ \\[2mm]
$\chi_{2,2}$ & $(p-1)/2$ & $(p-1)/2$ & $0$ & $Z^{(t)}$ & $Z$ & $Z^{(t)}$ & $Z$ \\[2mm]
\hline
$\chi _i |_U$ & $p+1$ & $p+1$ & $\zeta^{ij} + \zeta^{-ij}$ & 1 & 1 & 1 & 1
\end{tabular}\end{center}

\noindent{\bf(v):} Here $j$ is odd, and we get:

\begin{center}\begin{tabular}{c|l|l|l|l|l|l|l}
& 1 & $-1$ & $a^j$ & $b$ & $d$ & $-b^{-1}$ & $-d^{-1}$ \\
\hline
$\chi_{1,1}$ & $(p-1)/2$ & $-(p-1)/2$ & 0 & $Z$ & $Z^{(t)}$ & $-Z$ & $-Z^{(t)}$ \\[2mm]
$\chi_{2,1}$ & $(p-1)/2$ & $-(p-1)/2$ & 0 & $Z^{(t)}$ & $Z$ & $-Z^{(t)}$ & $-Z$ \\[2mm]
\hline
$\theta_j |_U$ & $p-1$ & $-(p-1)$ & 0 & -1 & -1 & 1 & 1
\end{tabular}\end{center}

\noindent{\bf (vi):} Here $j$ is even, and:

\begin{center}\begin{tabular}{c|l|l|l|l|l|l|l}
& 1 & $-1$ & $a^j$ & $b$ & $d$ & $-b^{-1}$ & $-d^{-1}$ \\
\hline
$\chi_{1,2}$ & $(p-1)/2$ & $(p-1)/2$ & $0$ & $Z$ & $Z^{(t)}$ & $Z$ &$ Z^{(t)}$ \\[2mm]
$\chi_{2,2}$ & $(p-1)/2$ & $(p-1)/2$ & $0$ & $Z^{(t)}$ & $Z$ & $Z^{(t)}$ & $Z$ \\[2mm]
\hline
$\theta_j |_U$ & $p-1$ & $p-1$ & 0 & -1 & -1 & -1 & -1
\end{tabular}\end{center}

\noindent{\bf (vii):} Because $p \equiv 1 \mod 4$, we get $(-1)^{(p-1)/2} = 1$. We also note that $\zeta^{(p-1)j/2} = (\zeta^{(p-1)/2})^j = (-1)^j$. Then we have: 

\begin{center}\begin{tabular}{c|l|l|l|l|l|l|l}
& 1 & $-1$ & $a^j$ & $b$ & $d$ & $-b^{-1}$ & $-d^{-1}$ \\
\hline
$\chi_{0,(p-1)/2}$ & 1 & 1 & $(-1)^j$ & 1 & 1 & 1 & $1$\\[2mm]
$\chi_{1,2}$ & $(p-1)/2$ & $(p-1)/2$ & $0$ & $Z$ & $Z^{(t)}$ & $Z$ &$ Z^{(t)}$ \\[2mm]
\hline
$\xi_1 |_U$ & $(p+1)/2$ & $(p+1)/2$ & $(-1)^j$ & $-Z^{(t)}$ & $-Z$ & $-Z^{(t)}$ & $-Z$
\end{tabular}\end{center}

\noindent{\bf (viii):} This case gives:

\begin{center}\begin{tabular}{c|l|l|l|l|l|l|l}
& 1 & $-1$ & $a^j$ & $b$ & $d$ & $-b^{-1}$ & $-d^{-1}$ \\
\hline
$\chi_{0,(p-1)/2}$ & 1 & 1 & $(-1)^j$ & 1 & 1 & 1 & 1 \\[2mm]
$\chi_{2,2}$ & $(p-1)/2$ & $(p-1)/2$ & $0$ & $Z^{(t)}$ & $Z$ & $Z^{(t)}$ & $Z$ \\[2mm]
\hline
$\xi_2 |_U$ & $(p+1)/2$ & $(p+1)/2$ & $(-1)^j$ & $-Z$ & $-Z^{(t)}$ & $-Z$ & $-Z^{(t)}$
\end{tabular}\end{center}

\noindent{\bf (ix):} For this case, we simply have:

\begin{center}\begin{tabular}{c|l|l|l|l|l|l|l}
& 1 & $-1$ & $a^j$ & $b$ & $d$ & $-b^{-1}$ & $-d^{-1}$ \\
\hline
$\chi_{1,1} = \eta_1 |_U$ & $(p-1)/2$ & $-(p-1)/2$ & 0 & $Z$ & $Z^{(t)}$ & $-Z$ & $-Z^{(t)}$
\end{tabular}\end{center}

\noindent{\bf (x):} It is seen that:

\begin{center}\begin{tabular}{c|l|l|l|l|l|l|l}
& 1 & $-1$ & $a^j$ & $b$ & $d$ & $-b^{-1}$ & $-d^{-1}$ \\
\hline
$\chi_{2,1} = \eta_2 |_U$ & $(p-1)/2$ & $-(p-1)/2$ & 0 & $Z^{(t)}$ & $Z$ & $-Z^{(t)}$ & $-Z$
\end{tabular}\end{center}

\noindent{\bf (xi):} Because $p \equiv 3 \mod 4$, we get $(-1)^{(p-1)/2} = -1$. Then we compute:

\begin{center}\begin{tabular}{c|l|l|l|l|l|l|l}
& 1 & $-1$ & $a^j$ & $b$ & $d$ & $-b^{-1}$ & $-d^{-1}$ \\
\hline
$\chi_{0,(p-1)/2}$ & 1 & $-1$ & $(-1)^j$ & 1 & 1 & 1 & $1$\\[2mm]
$\chi_{1,1}$ & $(p-1)/2$ & $-(p-1)/2$ & $0$ & $Z$ & $Z^{(t)}$ & $-Z$ & $-Z^{(t)}$ \\[2mm]
\hline
$\xi_1 ' |_U$ & $(p+1)/2$ & $-(p+1)/2$ & $(-1)^j$ & $-Z^{(t)}$ & $-Z$ & $Z^{(t)}$ & $Z$
\end{tabular}\end{center}

\noindent{\bf (xii):} For this case, we do the following:

\begin{center}\begin{tabular}{c|l|l|l|l|l|l|l}
& 1 & $-1$ & $a^j$ & $b$ & $d$ & $-b^{-1}$ & $-d^{-1}$ \\
\hline
$\chi_{0,(p-1)/2}$ & 1 & $-1$ & $(-1)^j$ & 1 & 1 & 1 & 1 \\[2mm]
$\chi_{2,1}$ & $(p-1)/2$ & $-(p-1)/2$ & $0$ & $Z^{(t)}$ & $Z$ & $-Z^{(t)}$ & $-Z$ \\[2mm]
\hline
$\xi_2 ' |_U$ & $(p+1)/2$ & $-(p+1)/2$ & $(-1)^j$ & $-Z$ & $-Z^{(t)}$ & $Z$ & $Z^{(t)}$
\end{tabular}\end{center}

\noindent{\bf (xiii):} We calculate the following restriction:

\begin{center}\begin{tabular}{c|l|l|l|l|l|l|l}
& 1 & $-1$ & $a^j$ & $b$ & $d$ & $-b^{-1}$ & $-d^{-1}$ \\
\hline
$\chi_{1,2} = \eta_1' |_U$ & $(p-1)/2$ & $(p-1)/2$ & 0 & $Z$ & $Z^{(t)}$ & $Z$ & $Z^{(t)}$
\end{tabular}\end{center}

\noindent{\bf (xiv):} Finally, we see:

\begin{center}\begin{tabular}{c|l|l|l|l|l|l|l}
& 1 & $-1$ & $a^j$ & $b$ & $d$ & $-b^{-1}$ & $-d^{-1}$ \\
\hline
$\chi_{2,2} = \eta_2 ' |_U$ & $(p-1)/2$ & $(p-1)/2$ & 0 & $Z^{(t)}$ & $Z$ & $Z^{(t)}$ & $Z$
\end{tabular}\end{center}

This completes consideration of all cases, and we conclude that $(G,U)$ is a strong Gelfand pair.
\end{proof}

Because $(G,U)$ is a strong Gelfand pair, we must now consider the maximal subgroups of $U$ to determine if any of them forms a strong Gelfand pair with $G$. 

\begin{prop}
For any proper subgroup $K \leq U$, $K\neq H_2$, the pair $(G, K)$ is not a strong Gelfand pair.

When $p \equiv 1 \mod 4$, the pair $(G, H_2)$ is not a strong Gelfand pair.

When $p \equiv 3 \mod 4$, the pair $(G, H_2)$ is a strong Gelfand pair.
\end{prop}

\begin{proof}
Recall that $U = \langle a, b\rangle = \langle a \rangle \rtimes \langle b \rangle$, so for any subgroup $K\leq U$ we have the following possibilities:

\noindent{\bf (i)} $p \nmid |K|$, in which case $K$ is conjugate to $\langle a^q \rangle$ for some $q$ dividing $p-1$. The maximal subgroups of $U$ containing $K$ are all conjugate to $A = \langle a \rangle$, the cyclic group of order $p-1$.
We see  that
$$\langle \varphi|_A, 1_A \rangle = \frac{1}{p-1}(p + p + 1 + \cdot \cdot \cdot + 1) = 3.$$ Thus, $(G,A)$ is not a strong Gelfand pair.\medskip

\noindent{\bf (ii)} $p \big{|} |K|$, in which case $K = \langle a^q \rangle \rtimes \langle b \rangle$ for some $q$ dividing $p-1$. The maximal subgroups are those with $q$ prime. We define $H_q = \langle a^q \rangle \rtimes \langle b \rangle$, which has index $q$ in $U$. Note that this corresponds with our earlier definition of $H_2$. Let $p_q = (p-1)/q$.

First, we look at $q >2$. Then $p_q \leq (p-1)/3 < (p-3)/2$, so referring to the bounds for Table \ref{CTG1} we consider the character $\chi_{p_q}$ of $G$. We find that the inner product $\langle \chi_{p_q}|_{H_q} , 1_{H_q} \rangle$ is

\begin{align*}
\frac{1}{p\cdot p_q} \Big( (p+1)1 &+ (p+1)1 + 1\cdot p_2 + 1\cdot p_2 + 1\cdot p_2 + 1\cdot p_2 + 2p(p_q - 2)\Big)
\\&= \frac{q}{p(p-1)}\left( 4p + 2p\left( \frac{p-1}{q}\right) - 4p\right) = 2.
\end{align*}
Thus, $(G, H_q)$ is not a strong Gelfand pair whenever $q > 2$.

When $q=2$, there are two cases. 
First, consider $p\equiv 1\mod 4$. As we can deduce from Table \ref{CTU}, each linear character of $H_2$ is the restriction of two different linear characters of $U$. In particular, $\chi_{0, (p-1)/4}|_{H_2} = \chi_{0, 3(p-1)/4}|_{H_2}$. From case (iv) of Proposition \ref{Prop1}, 
\begin{align*}
\chi_{(p-1)/4}|_{H_2} &= \chi_{0, (p-1)/4}|_{H_2} + \chi_{0, 3(p-1)/4}|_{H_2} + \chi_{1,2}|_{H_2} + \chi_{2,2}|_{H_2}
\\&= 2\chi_{0, (p-1)/4}|_{H_2} + \chi_{1,2}|_{H_2} + \chi_{2,2}|_{H_2},
\end{align*}
and we conclude that $(G,H_2)$ is not a strong Gelfand pair.

Finally, if $p \equiv 3 \mod 4$, we show that $(G, H_2)$ is a strong Gelfand pair. Because the subgroup $H_2$ has index 2 in $H$, the result in \cite[Prop. 20.10]{JaL} applies. Since $p \equiv 3 \mod 4$, $|H_2| = p(p-1)/2$ is odd, $-1 \notin H_2$. Thus, because $\chi (-1) \neq 0$ for all $\chi \in \hat H$, we have that each irreducible character of $H_2$ is simply the restriction of an irreducible character of $H$. It follows by Proposition \ref{Prop1} that $(G, H_2)$ is a strong Gelfand pair.

Because $(G, H_2)$ is a strong Gelfand pair when $p \equiv 3 \mod 4$, we must now consider subgroups of $H_2$. We show that no proper subgroup of $H_2$ forms a strong Gelfand pair with $G$. Because $H_2 = \langle a^2 \rangle \rtimes \langle b \rangle$, any subgroup of $H_2$ is of the form (i) $\langle a^{2k} \rangle \rtimes \langle b \rangle$ or (ii) $\langle a^{2k} \rangle$. The subgroups in (i) will be maximal when $k$ is prime; the maximal subgroup in (ii) is $\langle a^2 \rangle$.

\noindent {\bf (i)} If $k$ is an odd prime, we see that $\langle a^{2k} \rangle \rtimes \langle b \rangle$ will be a subgroup of $H_k$. Because $(G, H_k)$ is not a strong Gelfand pair, it follows that $\langle a^{2k} \rangle \rtimes \langle b \rangle$ also does not form a strong Gelfand pair with $G$.

If $k = 2$, then $4 | p-1$, and we have $H_4 = \langle a^4 \rangle \rtimes \langle b \rangle$.
Then $(p-1)/4 < (p-3)/2$, so referring to the bounds for Table \ref{CTG1} we consider the character $\chi_{p_4}$ of $G$. We find that the inner product $\langle \chi_{p_4}|_{H_4} , 1_{H_4} \rangle$ is

\begin{align*}
\frac{1}{p\cdot p_4} \Big( (p+1)1 &+ (p+1)1 + 1\cdot p_2 + 1\cdot p_2 + 1\cdot p_2 + 1\cdot p_2 + 2p(p_4 - 2)\Big)
\\&= \frac{4}{p(p-1)}\left( 4p + 2p\left( \frac{p-1}{4}\right) - 4p\right) = 2.
\end{align*}
So $(G, H_4)$ is not a strong Gelfand pair. 

\noindent {\bf (ii)} We see that $\langle a^{2} \rangle$ is a subgroup of $A$. Thus, by Lemma \ref{Lem2} and because $(G, A)$ is not a strong Gelfand pair, $(G, \langle a^2 \rangle)$ is also not a strong Gelfand pair.

Thus, we conclude that there is no subgroup $K$ of $H_2$ such that $(G, K)$ forms a strong Gelfand pair.

Applying Lemma \ref{Lem2}, we have shown that the only proper subgroup of $U$ which forms a strong Gelfand pair with $G$ is the subgroup $H_2$ when $p \equiv 3 \mod 4$.
\end{proof}

The above applies whenever $p\geq 13$, and so concludes the proof of Theorem \ref{Thm1} and Theorem \ref{Thm2}. \qed\medskip

To finish, we consider individually the cases $p = 2, 3, 5, 7, \text{ and }11$:

\begin{thm}\label{Thm3}

\noindent (i) $p=2$. In this case, $G \cong \Sigma_3$, and the strong Gelfand pairs are $(G, \mathcal C _2)$ and $(G, \mathcal C _3)$.

\noindent (ii) $p=3$. Then the strong Gelfand pairs are $(G, \mathcal C _3)$, $(G, \mathcal C _6)$, and $(G, U)$.

\noindent (iii) $p=5$. Then the strong Gelfand pairs are $(G, {\rm SL}(2,3))$ and $(G, U)$, where ${\rm SL}(2,3)\leq SL(2,5)$ is maximal.

\noindent (iv) $p=7$. Then the strong Gelfand pairs are $(G, U)$, $(G, H_2)$, and $(G, K)$, where $K$ is the preimage of $\Sigma_4$ in ${\rm PSL} (2,7)$. There are two non-conjugate copies of $K$ in $G$.

\noindent (v) $p=11$. Then the strong Gelfand pairs are $(G, H_2)$, $(G, U)$, and $(G, 2{\rm I)}$, where $$2{\rm I} = \langle r,s,t | r^2 = s^3 = t^5 = rst\rangle$$ is the binary icosahedral group. There are two non-conjugate copies of $2{\rm I}$ in $G$.

\end{thm}

\begin{proof}
These calculations were made using Magma \cite{Mag}. We give generating matrices for those subgroups which have not been defined previously:

\noindent {\bf (iii):}
${\rm SL}(2,3)$ as a subgroup of ${\rm SL}(2,5)$ is generated by
$$
\begin{pmatrix}
3 & 3 \\
4 & 1
\end{pmatrix},
\begin{pmatrix}
4 & 1 \\
3 & 1
\end{pmatrix},
\begin{pmatrix}
3 & 0 \\
1 & 2
\end{pmatrix},
\begin{pmatrix}
4 & 0 \\
0 & 4
\end{pmatrix}.
$$

\noindent {\bf (iv):}
One copy of $K$ is generated by
$$
\begin{pmatrix}
4 & 2 \\
2 & 3
\end{pmatrix},
\begin{pmatrix}
2 & 0 \\
6 & 4
\end{pmatrix},
\begin{pmatrix}
3 & 2 \\
2 & 4
\end{pmatrix},
\begin{pmatrix}
0 & 6 \\
1 & 0
\end{pmatrix},
\begin{pmatrix}
6 & 0 \\
0 & 6
\end{pmatrix};
$$
the other is generated by
$$
\begin{pmatrix}
3 & 6 \\
3 & 4
\end{pmatrix},
\begin{pmatrix}
5 & 0 \\
3 & 3
\end{pmatrix},
\begin{pmatrix}
2 & 6 \\
5 & 5
\end{pmatrix},
\begin{pmatrix}
4 & 4 \\
1 & 3
\end{pmatrix},
\begin{pmatrix}
6 & 0 \\
0 & 6
\end{pmatrix}.
$$

\noindent {\bf (v):}
The first copy of $2{\rm I}$ is generated by
$$
\begin{pmatrix}
7 & 2 \\
8 & 4
\end{pmatrix},
\begin{pmatrix}
6 & 2 \\
1 & 6
\end{pmatrix},
\begin{pmatrix}
10 & 0 \\
0 & 10
\end{pmatrix};
$$
the other is generated by
$$
\begin{pmatrix}
6 & 9 \\
2 & 5
\end{pmatrix},
\begin{pmatrix}
8 & 5 \\
4 & 4
\end{pmatrix},
\begin{pmatrix}
10 & 0 \\
0 & 10
\end{pmatrix}.
$$
These subgroups are also described in \cite{Fli}.
\end{proof}

\end{document}